\documentclass[12pt,a4paper]{amsart}%
\usepackage{amsmath}
\usepackage{amsfonts}
\usepackage{amssymb}
\usepackage{graphicx}%
\usepackage{comment}
\providecommand{\U}[1]{\protect\rule{.1in}{.1in}}
\newtheorem{theorem}{Theorem}

\newtheorem{example}[]{Example}

\newtheorem{proposition}[theorem]{Proposition}

\allowdisplaybreaks

\begin{document}

\title{Jittered sampling and probability measures}
\author[R. Bramati]{Roberto Bramati}
\address{Dipartimento di Ingegneria Gestionale, dell'Informazione e della Produzione,
Universit\`{a} degli Studi di Bergamo, Viale Marconi 5, 24044 Dalmine BG, Italy}
\email{roberto.bramati@unibg.it}
\author[L. Brandolini]{Luca Brandolini}
\address{Dipartimento di Ingegneria Gestionale, dell'Informazione e della Produzione,
Universit\`{a} degli Studi di Bergamo, Viale Marconi 5, 24044 Dalmine BG, Italy}
\email{luca.brandolini@unibg.it}
\author[G. Travaglini]{Giancarlo Travaglini}
\address{Dipartimento di Matematica e Applicazioni, Universit\`{a} di Milano-Bicocca,
Via Cozzi 55, 20125 Milano, Italy}
\email{giancarlo.travaglini@unimib.it}
\subjclass[2010]{Primary 11K38, 42B10}
\keywords{Geometric discrepancy, Fourier transforms, Dirac measures, jittered smapling, probability measures, zeros of Bessel functions}

\begin{abstract}
This paper investigates the discrepancy of a family of random sampling methods
obtained by perturbing the grid $\frac{1}{M}\mathbb{Z}^{d}\cap\left[
-1/2,1/2\right)^{d}$, where $M$ is a large positive integer. Parameterized
by an arbitrary probability measure $\mu$, this family encompasses
several classical methods for evaluating the quality of an $N$-point set in
$\mathbb{T}^{d}$, where $N=M^{d}$. We show that all probability
measures, except for Dirac measures, behave like the Lebesgue measure in the Monte Carlo
discrepancy. This represents a limiting case where the
measure $\mu$ depends on $M$. In this latter context, we prove that, up to a constant, the
lowest possible discrepancy is achieved when the support of $\mu$ has diameter $\leq c/M$, and that this upper bound is sharp.
\end{abstract}
\maketitle

The comparison between continuous and discrete objects is a key challenge in many fields,
including analysis, number theory, and numerical analysis. A classical problem
arises when considering a reasonably large family of geometric sub-regions
within a cube;\ the goal is to select $N$ points in the cube such that the
number of points falling in each sub-region closely approximates $N$ times its volume.

In the following, $d\geq2$ is the space dimension, $B=\left\{  t\in\mathbb{R}^{d}:\left\vert
t\right\vert \leq1\right\}$ is the unit ball, $\left\vert K\right\vert $ denotes the
volume of a measurable set $K$, and 
\[\chi
_{K}(t):=\left\{
\begin{array}
[c]{cc}%
1 & \text{if \ }t\in K\\
0 & \text{if \ }t\notin K
\end{array}
\right.\] is the characteristic (or indicator) function of $K$. We identify the $d$-dimensional torus $\mathbb{T}^{d}=\mathbb{R}^{d}/\mathbb{Z}^{d}$ with the cube
\[
Q=\left[  -1/2,1/2\right)^{d}.\] 
Here, $c,c_{1},\ldots$ denote positive constants that may change from step to
step. The notation $a_{N}=\mathcal{O}\left(  b_{N}\right)  $ means that there exists a
constant $c>0$ such that $\left\vert a_{N}\right\vert \leq c\left\vert
b_{N}\right\vert $, whereas $a_{N}=\mathit{o}\left(  b_{N}\right)  $ means
that $\left(  a_{N}/b_{N}\right)  \rightarrow0$.$\ $Moreover, $j_{N}\approx
k_{N}$ means that there exist two positive constants $c_{1},c_{2}$ such that
$\ c_{1}\left\vert j_{N}\right\vert \leq\left\vert k_{N}\right\vert \leq
c_{2}\left\vert j_{N}\right\vert $ for every $N$, whereas $j_{N}\sim k_{N}$
means $\left(  j_{N}/k_{N}\right)  \rightarrow1$. We write $a_{N}%
=\Omega\left(  b_{N}\right)  $ if there does not exist a constant $c>0$ such
that $\left\vert a_{N}\right\vert \leq c\left\vert b_{N}\right\vert $ for all
$N$ (that is, $a_{N}/b_{N}$ is unbounded).

Our starting point is the following result, which is due to W. Schmidt, H.
Montgomery, and J. Beck (see \cite{schmidt}, \cite{montgomery0},
\cite{Beck88}, \cite{montgomery}). It shows that a \textquotedblleft
large\textquotedblright\ error is unavoidable when evaluating the quality of
an arbitrary set of $N$ points with respect to the characteristic functions of balls. 

\begin{theorem}
\label{SMBT} There exists a positive
constant $c$ such that, for any choice of $N$\ points $z_{1},\ldots,z_{N}$ in
$\mathbb{T}^{d}$, we have%
\begin{equation}
\int_{0}^{1/2}\int_{\mathbb{T}^{d}}\left\vert \sum_{j=1}^{N}\chi_{rB-t}\left(
z_{j}\right)  -Nr^{d}\left\vert B\right\vert \right\vert ^{2}~dtdr\geq
cN^{1-1/d}. \label{SMB}%
\end{equation}
\end{theorem}

The integrand in \eqref{SMB} represents the \textquotedblleft
discrepancy\textquotedblright\ between the number of points $z_{j}$ falling in
the ball $rB-t$ and the expected value $Nr^{d}\left\vert B\right\vert $, that
is, $N$ times the volume of each ball $rB-t$.

The Schmidt--Montgomery--Beck theorem is naturally compared to a celebrated
result of K. Roth (see \cite{roth}), in which balls can be replaced by cubes (see also \cite{montgomery}), and the
lower bound $N^{1-1/d}$ in \eqref{SMB} becomes a logarithmic term. This
comparison suggests a geometric
insight into the problem of choosing $N$ points in a cube. 

We recall that H. Montgomery (see \cite{montgomery}) has derived both Theorem
\ref{SMBT} and Roth's theorem from a single argument originally due to J.
Cassels (see \cite{Cassels57}; see also \cite{BBG} for a very general form of
Cassels' result).

\medskip
There are several ways to prove that the lower bound $N^{1-1/d}$ in
\eqref{SMB} is sharp: the integer points approach of D. Kendall in
\cite{kendall} (see also \cite[p.210]{BGT14}); the jittered sampling method by
J. Beck and W. Chen (see e.g. \cite{Chen04}); and the combinatorial arguments
by J. Matousek \cite[5.4]{matousek} and W. Chen \cite{Chen96} (see also
\cite{BCCGT}).
In this paper we show that the upper bound $N^{1-1/d}$ applies to a large
family of examples arising from certain probability measures on $\mathbb{T}%
^{d}$.
\medskip

Let $\mu$ be a probability measure on $\mathbb{T}^{d}$. For every
$m\in\mathbb{Z}^{d}$ we consider the $m$-th Fourier coefficient of $\mu$
\[
\widehat{\mu}\left(  m\right)  :=\int_{\mathbb{T}^{d}}e^{-2\pi im\cdot
t}\ d\mu\left(  t\right).
\]
Notice that $\mu$ can also be regarded as a measure on $\mathbb{R}^d$ supported in the cube $Q$. Therefore we can also define, for every $\xi\in\mathbb{R}^d$,
\[
\widehat{\mu}\left(  \xi\right):= \int_{\mathbb{R}^{d}}e^{-2\pi i\xi\cdot
t}\ d\mu\left(  t\right) \ .
\]
The two definitions coincide whenever $\xi\in\mathbb{Z}^d$.
If $\mu$ is absolutely continuous with respect to the Lebesgue measure, that is, $\mu\left(  t\right)  =f\left(  t\right)  dt$, with $f\in L^{1}\left(
\mathbb{T}^{d}\right)  $ non-negative and $\int_{\mathbb{T}^{d}}f\left(
t\right)  \ dt=1$, then
\[
\widehat{\mu}\left(  m\right)  =\widehat{f}\left(  m\right)  :=\int
_{\mathbb{T}^{d}}f\left(  t\right)  e^{-2\pi im\cdot t}\ dt\ .
\]
For every positive integer $M$, let $N=M^{d}$ and consider the finite set
\begin{equation}
T_{N}=\left\{  t_{j}\right\}  _{j=1}^{N}:=\frac{1}{M}\mathbb{Z}^{d}\cap Q . \label{T_N}%
\end{equation}
For every $j=1,\ldots,N$, we translate $\mu$ by $t_{j}$, that is we consider
the measure $\mu_{j}$ defined by%
\[
\int_{\mathbb{T}^{d}}g\left(  t\right)  \ d\mu_{j}\left(  t\right)
:=\int_{\mathbb{T}^{d}}g\left(  t-t_{j}\right)  \ d\mu\left(  t\right),
\]
where $g\left(  t\right)  $ is any continuous function on $\mathbb{T}^{d}$.

Let $V_{N}=\{v_{1},\ldots,v_{N}\}\subset\mathbb{T}^{d}$. For $t\in
\mathbb{T}^{d}$ and $0<r<1/2$, we introduce the \textit{discrepancy function}%
\begin{equation}
D_{V_{N}}\left(  t\right)  :=\sum_{j=1}^N\chi_{rB-t}\left(  v_{j}\right)-Nr^{d}\left\vert B\right\vert   \ , \label{asterisco}%
\end{equation}
which has the Fourier series%
\[
\sum_{0\neq m\in\mathbb{Z}^{d}}\left(  \sum_{j=1}^{N}e^{2\pi im\cdot v_{j}%
}\right)  \widehat{\chi}_{rB}\left(  m\right)  e^{2\pi im\cdot t}%
\]
(see e.g. \cite[p.205]{BGT14}).

\medskip

In this paper, we study the following \textquotedblleft generalized
discrepancy\textquotedblright, introduced in \cite[p.290]{CT09} (see also
\cite{BChGT}, \cite[p.213]{BGT14}, \cite{BCCGT} and \cite{TRA14}):%
\begin{equation}
D_{\mu}\left(  N\right)  :=\left\{  \int_{\mathbb{T}^{d}}\cdots
\int_{\mathbb{T}^{d}}\int_{\mathbb{T}^{d}}D_{V_{N}}^{2}\left(  t\right)
\ dtd\mu_{1}\left(  v_{1}\right)  \cdots d\mu_{N}\left(  v_{N}\right)
\right\}  ^{1/2}\ , \label{def_mu}%
\end{equation}
where $V_N=\{v_{1}%
,\ldots,v_{N}\}$ denotes the set of integration variables, and $D_{V_N}$ is the associated discrepancy. Here, we replace the Lebesgue measure used in the
Monte Carlo method with an arbitrary probability measure $\mu$ on
$\mathbb{T}^{d}$. This step warrants an explanation:\ since a probability measure
$\mu$ distinct from the Lebesgue measure $dt$ is not uniform, we must
determine its placement within $\mathbb{T}^{d}$. Specifically, we position
$N=M^{d}$ translated copies of $\mu$, each centered at one of the points in $T_{N}$ (see \eqref{T_N}). We call \eqref{def_mu} the 
generalized discrepancy because it encompasses several classical
methods for evaluating the quality of a set of $N$ points in $\mathbb{T}^{d}$
with respect to a family of balls. Specifically, we highlight the following
special cases.

\begin{example} $\mu\left(  t\right)  =dt$ (the Lebesgue measure). Here each point
$v_{j}$ is chosen independently and uniformly at random in $\mathbb{T}^{d}$.
In this setting $D_{\mu}\left(  N\right)  $ corresponds to the \textit{Monte
Carlo discrepancy}, and it can be easily shown (see e.g. \cite[p.205]{TRA14})
that\textit{ }%
\[
D_{dt}^{2}\left(  N\right)  =N\left(  \left\vert rB\right\vert -\left\vert
rB\right\vert ^{2}\right)  \ .
\]
\end{example}

\begin{example}\label{ex:jitt}
$\mu(t)=\nu_M\left(  t\right)  :=N\chi_{M^{-1}Q}\left(  t\right)  dt$ (the normalized characteristic
function of the small cube $M^{-1}Q$). Each point $v_{j}$ is chosen at random inside the cube $M^{-1}Q+t_{j}$ (see \eqref{T_N} above for the choice of the points $t_{j}$).
In this case $D_{\nu_M}\left(  N\right)  $ is sometimes known as
\textit{jittered sampling discrepancy} and it is known (see
\cite[p.290]{CT09}) that there exists a positive constant $c$ such that
\[
D_{\nu_M}^{2}\left(  N\right)  \sim cN^{1-1/d}\ .
\]
The jittered sampling discrepancy has been studied in several papers. See, among others, \cite{BCT25,Doe22,KP22,PS16,PRS18}. 
\end{example}

\begin{example} $\mu=\delta_{0}$ (the Dirac measure at the origin). Hence we are sampling $\chi_{rB-t}$ over the grid $T_N$ defined in \eqref{T_N}. In this case, $D_{\delta_0
}\left(  N\right)  $ can be termed the \textit{grid discrepancy}. 
This is a slightly less general variant of the following well-known problem on integer points in $\mathbb{R}^d$ (in our example, the dilation parameter is restricted to integer values, whereas in the classical problem it is allowed to be real).
Let $R$ be a large positive number, and let
$RB=\left\{  t\in\mathbb{R}^{d}:\left\vert t\right\vert \leq R\right\}$. In
the case $d=2$, the best possible estimate of the discrepancy
\[
D_{R}:=\sum_{m\in\mathbb{Z}^{2}}\chi_{RB}\left(  m\right)-\pi R^{2}
\]
is a classical and still unsolved problem in number theory, known as the 
\textquotedblleft Gauss circle problem\textquotedblright\ (see e.g.
\cite{Kratzel}). D. Kendall (see \cite{kendall}, see also \cite[p.188]{BGT14})
observed that if we shift $B$ and consider, for every $t\in\mathbb{T}^{d}$,
the discrepancy function%
\[
D_{R}\left(  t\right)  :=\sum_{m\in
\mathbb{Z}^{d}}\chi_{RB-t}\left(  m\right)-R^{d}\left\vert B\right\vert   \ ,
\]
then
\begin{equation}
\int_{\mathbb{T}^{d}}\left\vert D_{R}\left(  t\right)  \right\vert
^{2}\ dt\leq cR^{d-1}\ , \label{Ken}%
\end{equation}
so that, in this example, (\ref{def_mu}) essentially reduces to the LHS in (\ref{Ken}). We recall that it is not possible to replace the power \ $d-1 $ \ in \eqref{Ken}
with a smaller exponent, but we point out that a converse inequality, such as
$\int_{\mathbb{T}^{d}}\left\vert D_{R}\left(  t\right)  \right\vert
^{2}\ dt\geq c_{1}R^{d-1}$, is not always true (see e.g. \cite{PS01},
\cite[p.188]{BGT14}, \cite{BCGT15}, see also \cite{BGGM,montgomery,TRA14,TT} for related problems). In order to obtain a converse to 
\eqref{Ken} we need to integrate over the radius of the ball,
which is precisely what we have seen in \eqref{SMB}.
\end{example}

Our first main result  
demonstrates that $D_{\mu}^{2}\left(  N\right)  \approx N$ if and only if
$\mu$ is not a Dirac measure $\delta_{t_{0}}$ for some $t_{0}\in\mathbb{T}^{d}%
$. See Theorem \ref{indip eps} below.

This result can be viewed as a  \textquotedblleft limiting case\textquotedblright\ (with $\mu$ independent of $M$) of the
following framework: we return to the jittered sampling measure \  $\nu_M\left(  t\right)
=N\chi_{M^{-1}Q%
}\left(  t\right)  dt$ \ discussed above and we replace the measures $\nu_M$ with arbitrary probability
measures $\mu_M$. We obtain the upper bound\  $D_{\mu_M}%
^{2}\left(  N\right)  \leq cN^{1-1/d}$ \ 
whenever the support of $\mu_M$ is suitably small, and mixed results in the
complementary cases. 
See Theorem \ref{dip eps}, Theorem \ref{thm:largeSupport} and Proposition \ref{prop:no_rescaling} below. 

In short, 
in view of Theorem \ref{SMBT}, we are interested
in studying the probability measures $\mu_M$ that satisfy $D_{\mu_M}%
^{2}\left(  N\right)  \leq cN^{1-1/d}$. 
\medskip

We have the following results.

\begin{theorem}
\label{indip eps}Let $\mu$ be a probability measure on $\mathbb{T}^d$, different from any Dirac measure  $\delta_{t_{0}}$. Then
$D_{\mu}^{2}\left(  N\right)  \approx N$.
\end{theorem}

\begin{theorem}
\label{dip eps}Let $\mu_M$ be a sequence of probability measures on $\mathbb{T}^d$. Assume that every measure $\mu_M$ has 
support in a ball of radius $\varepsilon_M=\mathcal{O}\left(M^{-1}\right)$. Then $D_{\mu_M}^{2}\left(  N\right)  \leq cN^{1-1/d}$.
\end{theorem}

In the next theorem, we show that the assumption $\varepsilon_M=\mathcal{O}\left(M^{-1}\right)$ cannot be improved. Let $\mu$ be a probability measure on $Q$. For every
$0<\varepsilon<1$ let $\mu^{\left(\varepsilon\right)}$ be the probability measure
defined by
\[
\int_{\mathbb{R}^{d}}f\left(x\right)d\mu^{\left(\varepsilon\right)}\left(x\right):=\int_{\mathbb{R}^{d}}f\left(\varepsilon x\right)d\mu\left(x\right).
\]
Therefore
\[
\widehat{\mu^{\left(\varepsilon\right)}}\left(\xi\right)=\int_{\mathbb{R}^{d}}e^{-2\pi ix\cdot\xi}d\mu^{\left(\varepsilon\right)}\left(x\right)=\int_{\mathbb{R}^{d}}e^{-2\pi ix\cdot\varepsilon\xi}d\mu\left(x\right)=\widehat{\mu}\left(\varepsilon\xi\right).
\]

\begin{theorem}\label{thm:largeSupport}
Let $\mu$ be a probability measure on $\mathbb{T}^d$ that does not coincide with a Dirac measure $\delta_{t_0}$. Then there exists a positive constant $c$ such that, for every
$0<\varepsilon<1$ and every $N$,
\[
D^{2}_{\mu^{\left(\varepsilon\right)}}\left(N\right)\geq cN\varepsilon.
\]
In particular, choosing  \  $\varepsilon_M=\Omega\left(M^{-1}\right)$ \ and $\mu_M=\mu^{\left(\varepsilon_M\right)}$  we have
\[
D^{2}_{\mu_M}\left(N\right)=\Omega\left(N^{1-1/d}\right).
\]
\end{theorem}
The assumptions of Theorem \ref{thm:largeSupport} encompass different types of measures. We give three examples, where we assume that, as $M\rightarrow+\infty$, we have $\varepsilon_M \rightarrow0$, and $\varepsilon_M M$ is  unbounded, that is, $\varepsilon_M=\Omega\left(  M^{-1}\right)$.
\begin{enumerate}
\item The normalized characteristic function of a ball of radius $\varepsilon_M$. We recall that $B=\left\{  t\in\mathbb{R}^{d}:\left\vert t\right\vert
\leq1\right\}  $, and we choose the probability measure%
\[
\mu_M\left(  t\right)  =\frac{\Gamma\left(  1+d/2\right)  }{\varepsilon_M^{d}%
\pi^{d/2}}\chi_{\varepsilon_M B}\left(  t\right)  dt. 
\]
\item The normalized uniform  measure on the sphere of radius $\varepsilon_M$ 
$\left\{  t\in\mathbb{R}^{d}:\left\vert t\right\vert =\varepsilon_M\right\}$. 
\item A half sum of two Dirac measures at distance $\varepsilon_M$ from each other,
\[
\mu_M=\frac{1}{2}\left(  \delta_{0}+\delta_{\varepsilon_M v}\right),  \]
where $v$ is a fixed unit vector.
\end{enumerate}

We point out that Theorem \ref{dip eps} holds for any sequence of measures $\mu_M$ supported in a ball of radius $\leq c/M$. In particular, the \textquotedblleft shapes\textquotedblright\  of the measures $\mu_M$ may change drastically as their supports shrink. In contrast, Theorem \ref{thm:largeSupport} considers only the particular case of a sequence of measures obtained by rescaling a given measure $\mu$. The following proposition shows that Theorem \ref{thm:largeSupport} cannot be extended to every sequence of measures.

\begin{proposition}\label{prop:no_rescaling}
Let $\varepsilon_M$ satisfy $M^{-1}<\varepsilon_M<1$ (this assumption covers the case $\varepsilon_M=\Omega(M^{-1})$). Then there exists a sequence of measures $\mu_M$ on $\mathbb{T}^d$, each supported in a ball of radius $\varepsilon_M$, such that 
\[
D^2_{\mu_M}\left(N\right)\leq cN^{1-1/d}.
\]
\end{proposition}

The proofs of the above theorems depend on the following identities
(see \cite{CT09} and \cite[p.215]{BGT14}), which hold true for every choice of
the probability measure $\mu$ (see \eqref{T_N} and \eqref{asterisco} for the definition of $D_{T_{N}}$):
\begin{align}
 D_{\mu}^{2}\left(  N\right)  =&N\left(  \left\vert rB\right\vert
-\left\Vert \chi_{rB}\ast\mu\right\Vert _{L^{2}\left(  \mathbb{T}^{d}\right)
}^{2}\right)  +\left\Vert D_{T_{N}}\ast\mu\right\Vert _{L^{2}\left(
\mathbb{T}^{d}\right)  }^{2}\label{discr}\\
 =&N\sum_{0\neq m\in\mathbb{Z}^{d}}\left(  1-\left\vert \widehat{\mu}\left(
m\right)  \right\vert ^{2}\right)  \left\vert \widehat{\chi}_{rB}\left(
m\right)  \right\vert ^{2}\nonumber\\
&+N^{2}\sum_{0\neq h\in\mathbb{Z}^{d}}\left\vert
\widehat{\chi}_{rB}\left(  Mh\right)  \right\vert ^{2}\left\vert \widehat{\mu
}\left(  Mh\right)  \right\vert ^{2} \ .\nonumber
\end{align}

\medskip

We recall that $\widehat{\chi}_{rB}$ (see, e.g., \cite{Watson} or \cite{Guido})
can be written in terms of the Bessel function of the first kind $J_{d/2}$:%
\begin{align}
\widehat{\chi}_{rB}\left(  m\right)   &  =r^{d}\widehat{\chi}_{B}\left(
rm\right)  =r^{d}\left\vert rm\right\vert ^{-d/2}J_{d/2}\left(  2\pi
r\left\vert m\right\vert \right). \label{bes}
\end{align}
We also recall that, for $\left\vert m\right\vert \geq 1$, we have
\begin{equation}
\widehat{\chi}_{rB}\left(  m\right)  =\frac{r^{\left(  d-1\right)  /2}}{\pi
\left\vert m\right\vert ^{\left(  d+1\right)  /2}}\cos\left(  2\pi
r\left\vert m\right\vert -\frac{d+1}{4}\pi\right)  +\mathcal{O}\left(r
\left\vert m\right\vert ^{-(d+3)/2}\right).  \label{Bes}%
\end{equation}
Usually \eqref{Bes} is stated for large values of $|m|$, however, a suitable choice of the implicit constant in $\mathcal{O}$ makes it hold for every $|m|\geq1$.

\begin{proof}
[Proof of Theorem \ref{indip eps}]For every $m\in\mathbb{Z}^{d}$,
\eqref{discr} and \eqref{bes} yield
\[
D_{\mu}^{2}\left(  N\right)  \geq cNr^{d}\sum_{0\neq m\in\mathbb{Z}^{d}%
}\left(  1-\left\vert \widehat{\mu}\left(  m\right)  \right\vert ^{2}\right)
\left\vert m\right\vert ^{-d}\left\vert J_{d/2}\left(  2\pi r\left\vert
m\right\vert \right)  \right\vert ^{2}\ .
\]
To prove that $D_{\mu}^{2}\left(  N\right)  \geq cN$ whenever $\mu$ is
different from a Dirac measure $\delta_{t_{0}}$ (for any $t_{0}\in\mathbb{T}%
^{d}$), it suffices to prove the existence of $\widetilde{m}\in\mathbb{Z}^{d}$
such that
\begin{equation}
J_{d/2}\left(  2\pi r\left\vert \widetilde{m}\right\vert \right)
\neq0\text{\ \ \ \ \ \ and\ \ \ \ \ \ }\left\vert \widehat{\mu}\left(
\widetilde{m}\right)  \right\vert <1\ . \label{both}%
\end{equation}
First we look at the Bessel function and we write the positive solutions of
the equation $J_{d/2}\left(  x\right)  =0$ as%
\[
j_{d/2,1}<j_{d/2,2}<j_{d/2,3}<\ \ldots\ <j_{d/2,n}<\ \ldots
\]
Equation \eqref{Bes} shows that, going to infinity, the zeros of $J_{d/2}\left(  x\right)
$ must be \textquotedblleft close\textquotedblright\ to the zeros of \ 
$\cos\left(x-\left(d+1\right)\pi/4\right)  $. More precisely, McMahon's
formula (see e.g. \cite[9.5.12]{ABST} or \cite[p.506]{Watson}) shows that, for
large positive integers $n$, we have

\begin{align}
j_{d/2,n} &= \left( n+\frac{d-1}{4}\right) \pi - \frac{d^{2}-1}{8\pi\left( n+\frac{d-1}{4}\right) } - \frac{\left( d^{2}-1\right) \left( 7d^{2}-31\right) }{384\pi^{3}\left( n+\frac{d-1}{4}\right)^{3}} \label{McM} \\
& \quad + \mathcal{O}\left( n^{-5}\right) \ . \nonumber
\end{align}
It is clear that the sequence $j_{d/2,n}$ is not an arithmetic progression.
More specifically, we will prove that it cannot even contain an arithmetic
progression. Indeed, assuming that the subsequence $j_{d/2,n_{k}}$ is an arithmetic progression, then \eqref{McM} yields (writing $\eta:=\left(
d-1\right)  \pi/4$)%
\begin{align}
0    = & \ j_{d/2,n_{k+1}}-2j_{d/2,n_{k}}+j_{d/2,n_{k-1}}\label{long}\\
  =& \ \pi n_{k+1}+\eta-\frac{d^{2}-1}{8\pi\left(  n_{k+1}+\eta\right)  }%
-\frac{\left(  d^{2}-1\right)  \left(  7d^{2}-31\right)  }{384\pi^{3}\left(
n_{k+1}+\eta\right)  ^{3}}\nonumber\\
&  -2\pi n_{k}-2\eta+2\frac{d^{2}-1}{8\pi\left(  n_{k}+\eta\right)  }%
+2\frac{\left(  d^{2}-1\right)  \left(  7d^{2}-31\right)  }{384\pi^{3}\left(
n_{k}+\eta\right)  ^{3}}\nonumber\\
&  +\pi n_{k-1}+\eta-\frac{d^{2}-1}{8\pi\left(  n_{k-1}+\eta\right)  }%
-\frac{\left(  d^{2}-1\right)  \left(  7d^{2}-31\right)  }{384\pi^{3}\left(
n_{k-1}+\eta\right)  ^{3}}+\mathcal{O}\left(  n_{k-1}^{-5}\right) \nonumber\\
= & \ \pi\left(  n_{k+1}-2n_{k}+n_{k-1}\right) \nonumber\\
&  +\frac{d^{2}-1}{8\pi}\left(  \frac{-1}{n_{k+1}+\eta}+\frac{2}{n_{k}+\eta
}-\frac{1}{n_{k-1}+\eta}\right) \nonumber\\
&  +\frac{\left(  d^{2}-1\right)  \left(  7d^{2}-31\right)  }{384\pi^{3}%
}\left(  \frac{-1}{\left(  n_{k+1}+\eta\right)  ^{3}}+\frac{2}{\left(
n_{k}+\eta\right)  ^{3}}-\frac{1}{\left(  n_{k-1}+\eta\right)  ^{3}}\right)
\nonumber\\
&  +\mathcal{O}\left(  n_{k-1}^{-5}\right)  \ .\nonumber
\end{align}
We consider two cases. First assume that $n_{k}$ is an arithmetic progression,
then the same is true for $h_{k}:=n_{k}+\eta$ and we write $h_k=\alpha k+\beta$. Then
\[
\pi\left(  n_{k+1}-2n_{k}+n_{k-1}\right)=0.
\]
Also
\begin{align*}
\frac{-1}{h_{k+1}}+\frac{2}{h_{k}}-\frac{1}{h_{k-1}}= &	\ \frac{-1}{\alpha k+\beta+\alpha}+\frac{2}{\alpha k+\beta}-\frac{1}{\alpha k+\beta+\alpha}\\
= &\ 	-\frac{2}{\alpha k^{3}}+\mathcal O\left(\frac{1}{k^{4}}\right).
\end{align*}
A direct, though somewhat lengthy, computation gives
\[
-\frac{1}{h^{3}_{k+1}}+\frac{2}{h^{3}_{k}}-\frac{1}{h^{3}_{k-1}}=\mathcal O\left(\frac{1}{k^4}\right).
\]
Then \eqref{long}
becomes%
\[
0    = \frac{c}{k^3}+\mathcal{O}\left(\frac{1}{k^4}\right) 
\]
which is impossible. Now assume that $n_{k}$ is not an arithmetic progression.
Since the numbers $n_{k}$ are integers, then, for infinitely many of them, we
have $\left\vert n_{k+1}-2n_{k}+n_{k-1}\right\vert \geq1$. Therefore
, \eqref{long} yields a contradiction. So far we have proved that the set of zeros of the equation $J_{d/2}\left(  x\right)  =0$ does not contain an
arithmetic progression. We have to go back to the equation
\begin{equation}
J_{d/2}\left(  2\pi r\left\vert m\right\vert \right)  =0\ , \label{mod}%
\end{equation}
where $m\in\mathbb{Z}^{d}$. We claim that no $d$-dimensional arithmetic
progression of the particular form
\begin{equation}
m_k=\left(  k+s\right)  m_{0}\ , \label{quasi}%
\end{equation}
with $m_{0}\in\mathbb{Z}^{d}$, $s\in\mathbb{N}$, and $k=1,2,\ldots$ can satisfy \eqref{mod}. Indeed, assume that%
\[
0=J_{d/2}\left(  2\pi r\left\vert m_k\right\vert \right)  =J_{d/2}\left(  2\pi
r\left(  k+s\right)  \left\vert m_{0}\right\vert \right)  \ .
\]
Then the arithmetic progression
\[
x_{k}=2\pi rk\left\vert m_{0}\right\vert +2\pi rs\left\vert m_{0}\right\vert
\]
satisfies $J_{d/2}\left(x_{k}\right)=0$, which contradicts the previous remark on the zeros of $J_{d/2}$. 

We now turn to the second condition in \eqref{both}. A general
result proved by E. Hewitt and H. Rubin on locally compact abelian groups
(\cite[Theorems 1.1]{HR}) implies that, for every probability measure $\mu$ on
$\mathbb{T}^{d}$, the two sets
\begin{equation}
\begin{aligned}
H_{0}\left(\mu\right):=\left\{  m\in\mathbb{Z}^{d}:\widehat{\mu}\left(
m\right)=1\right\}, \\
H\left(  \mu\right)  :=\left\{
m\in\mathbb{Z}^{d}:\left\vert \widehat{\mu}\left(  m\right)  \right\vert
=1\right\} 
\end{aligned}
\label{acca}%
\end{equation}
are subgroups of $\mathbb{Z}^{d}$, hence sublattices. In the case of $\mathbb{T}^d$, this result is also a consequence of the argument at the end of the proof of Lemma 5 in \cite{BChGT}. 

We claim that if $\mu$
is not the Dirac measure at a point $t_{0}\in\mathbb{T}^{d}$, then $H\left(\mu\right)$ is a proper subset of
$\mathbb{Z}^{d}$. Indeed, it is well known that $H_{0}\left(\mu\right) 
=\mathbb{Z}^{d}$ if and only if $\mu$ is the Dirac measure $\delta_{0}$. Then \cite[Theorems 1.1]{HR} or a modification of a known argument (see \cite[Lemma 5]{BChGT}) yield that $H\left(\mu\right)
=\mathbb{Z}^{d}$ if and only if $\mu$ is a Dirac measure $\delta_{t_{0}}$ at a
point $t_{0}\in\mathbb{T}^{d}$. Indeed, assume that $\left\vert \widehat{\mu}\left(
m\right)  \right\vert =1$ for every $m\in\mathbb{Z}^{d}$. Then there exists a real number $\alpha$ (depending on $m$) such that 
$\widehat{\mu}\left(  m\right)  =e^{-2\pi i\alpha}$. Let$\ e_{1},e_{2},\ldots,e_{d}$ be the standard
basis of $\mathbb{R}^{d}$. Then for every $j=1,2,\ldots,d$ there exists
$x_{j}\in\mathbb{R}$ such that, writing $t=\left(  t_{1},t_{2},\ldots
,t_{d}\right) $, we have
\[
\widehat{\mu}\left(  e_{j}\right)  =\int_{\mathbb{T}^{d}}e^{-2\pi ie_{j}\cdot
t}\ d\mu\left(  t\right)  =\int_{\mathbb{T}^{d}}e^{-2\pi it_{j}}\ d\mu\left(
t\right)  =e^{-2\pi ix_{j}}\ .
\]
Hence,
\[
\int_{\mathbb{T}^{d}}e^{-2\pi i\left(t_{j}-x_{j}\right)}\ d\mu\left(t\right) =  1 \ .
\]
Since $\mu$ is a probability measure and $|e^{-2\pi i\left(  t_{j}%
-x_{j}\right)}|=1$, we have that $e^{-2\pi i\left(  t_{j}%
-x_{j}\right)}=1$ almost everywhere with respect to $\mu$. Then
$\mu=\delta_{x}$, where $x=\left(  x_{1},x_{2},\ldots,x_{d}\right)$. Hence,
if $\mu$ is not a Dirac measure, then there exists $\widetilde{m}\in
\mathbb{Z}^{d}$ such that $\left\vert \widehat{\mu}\left(  \widetilde
{m}\right)  \right\vert <1$. That is, $\widetilde{m}\in\mathbb{Z}%
^{d}\backslash H\left( \mu\right)  $. We claim that $\mathbb{Z}%
^{d}\backslash H\left( \mu\right)  $ contains an arithmetic progression of
the particular form \eqref{quasi}. Indeed, at least one of the unit vectors $e_{1}%
,e_{2},\ldots,e_{d}$ (say $e_{1}$) does not belong to the proper  subgroup $H\left(
\mu\right)  $, otherwise $H\left( \mu\right)  =\mathbb{Z}^{d}$. If
all the elements of the form $ke_{1}$ ($k\in\mathbb{N}$) do not belong to
$H\left( \mu\right)  $, then we are done. Otherwise, let $k_{0}$ be the
first positive integer such that $\left\vert \widehat{\mu}\left(  k_{0}%
e_{1}\right)  \right\vert =1$, that is, $k_{0}e_{1}\in H\left(  \mu\right)  $.
Then $\left(  k_{0}+1\right)  e_{1}\notin H\left( \mu\right)  $, otherwise
the difference
\[
\left(  k_{0}+1\right)  e_{1}-k_{0}e_{1}=e_{1}\in H\left( \mu\right)  \ .
\]
In the same way, we see that
\[
K:=\left\{  ke_{1}\in\mathbb{Z}^{d}:k\text{ is not a positive integer multiple
of }k_{0}\right\}
\]
satisfies%
\[
K\cap H\left(\mu\right)  =\varnothing\ .
\]
This yields that $K$ contains an arithmetic progression of the form
\eqref{quasi}. For example, in the simplest case $k_{0}=2$, the set 
\[
K=\left\{  ke_{1}\in\mathbb{Z}^{d}:k\text{ odd}\right\}
\]
is an arithmetic progression of the form \eqref{quasi}.

Now we have proved that there exists $\widetilde{m}$ such that \eqref{both}
holds. Hence%
\[
D_{\mu}^{2}\left(  N\right)  \geq N\left(  1-\left\vert \widehat{\mu}\left(
\widetilde{m}\right)  \right\vert ^{2}\right)  \left\vert \widehat{\chi}%
_{rB}\left(  \widetilde{m}\right)  \right\vert ^{2}\geq cN\ .
\]

To prove the reverse inequality $D_{\mu}^{2}\left(  N\right)  \leq c_{1}N$,
we first recall that $\left\vert \widehat{\mu}\left(  m\right)  \right\vert
\leq1$ and that \eqref{Bes} yields $\left\vert \widehat{\chi}_{rB}\left(
m\right)  \right\vert ^2\leq c\left\vert m\right\vert ^{-d-1}$. Then \eqref{discr} gives 
\begin{align}
D_{\mu}^{2}\left(  N\right)   &  \leq N\sum_{0\neq m\in\mathbb{Z}^{d}%
}\left\vert \widehat{\chi}_{rB}\left(  m\right)  \right\vert ^{2}+N^{2}%
\sum_{0\neq h\in\mathbb{Z}^{d}}\left\vert \widehat{\chi}_{rB}\left(
Mh\right)  \right\vert ^{2}\left\vert \widehat{\mu}\left(  Mh\right)
\right\vert ^{2}\nonumber\\
&  \leq cN\sum_{0\neq m\in\mathbb{Z}^{d}}\left\vert m\right\vert
^{-d-1}+cN^{2}\sum_{0\neq h\in\mathbb{Z}^{d}}\left(  M\left\vert h\right\vert
\right)  ^{-d-1}\nonumber\\
&  =cN\sum_{0\neq m\in\mathbb{Z}^{d}}\left\vert m\right\vert ^{-d-1}%
+cN^{1-1/d}\sum_{0\neq h\in\mathbb{Z}^{d}}\left\vert h\right\vert
^{-d-1}\label{fatto}\\
&  \leq cN\int_{\left\{  t\in\mathbb{R}^{d}:\left\vert t\right\vert
\geq1\right\}  }\left\vert t\right\vert ^{-d-1}\ dt\leq cN\ .\nonumber
\end{align}

\end{proof}

\medskip

\begin{proof}
[Proof of Theorem \ref{dip eps}]
\eqref{discr} shows that $D_{\mu_M}^{2}\left(N\right) $ depends only on the absolute value of $\widehat{\mu}_M$. Hence we can assume that
$\mu_M$ is a probability measure on $Q$, with support in a ball of radius $\varepsilon_M$ centered at
the origin. We have seen in \eqref{fatto} that%
\[
N^{2}\sum_{0\neq h\in\mathbb{Z}^{d}}\left\vert \widehat{\chi}_{rB}\left(
Mh\right)  \right\vert ^{2}\left\vert \widehat{\mu}_M\left(  Mh\right)
\right\vert ^{2}\leq cN^{1-1/d}\ .
\]
Hence it suffices to bound the term%
\[
N\left(  \left\vert rB\right\vert -\left\Vert \chi_{rB}\ast\mu_M\right\Vert
_{L^{2}\left(  \mathbb{T}^{d}\right)  }^{2}\right)  =N\sum_{0\neq
m\in\mathbb{Z}^{d}}\left(  1-\left\vert \widehat{\mu}_M\left(  m\right)
\right\vert ^{2}\right)  \left\vert \widehat{\chi}_{rB}\left(  m\right)
\right\vert ^{2}\ .
\]
Following a known argument (see e.g. \cite[p. 222]{TRA14}) we observe that
\[
\left(  \chi_{rB}\ast\mu_M\right)  \left(  t\right)  =\chi_{rB}\left(  t\right)
\]
for every $t\in\mathbb{T}^{d}$ outside the $d$-dimensional annulus determined
by the two balls
\[
\left\{  t\in\mathbb{R}^{d}:\left\vert t\right\vert \leq r\pm\varepsilon_M
\sqrt{d}/2\right\}  \ .
\]
This annulus has volume $\leq c\varepsilon_M$. Since \ $\left\Vert \chi_{rB}\ast\mu_M\right\Vert
_{L^{2}\left(  \mathbb{T}^{d}\right)  }^{2} \leq \left\vert rB\right\vert$ \ and $\varepsilon_M\leq cN^{-1/d}$, then 
\[
N\left(  \left\vert rB\right\vert -\left\Vert \chi_{rB}\ast\mu_M\right\Vert
_{L^{2}\left(  \mathbb{T}^{d}\right)  }^{2}\right)  \leq cN\varepsilon_M\leq
cN^{1-1/d}\ .
\]

\end{proof}

\begin{proof}
[Proof of Theorem \ref{thm:largeSupport}]
Let $x_{0}=\int_{\mathbb{R}^{d}}xd\mu\left(x\right)$ be the center
of mass of $\mu$ and let $\mu_{x_{0}}$ the measure translated
by $x_{0}$, that is 
\[
\int f\left(x\right)d\mu_{x_{0}}\left(x\right)=\int f\left(x-x_{0}\right)d\mu\left(x\right).
\]
Since $D_{\mu}^{2}\left(  N\right) $ depends only on the absolute value of $\widehat{\mu}$, then  $D^{2}_{\mu}\left(N\right)=D^{2}_{\mu_{x_{0}}}\left(N\right)$. Then 
we can assume $x_{0}=0$. Since $\mu$ has compact support, then $\widehat{\mu}\left(\xi\right)$
is smooth and we have the following Taylor expansion (see \cite[Lemma 5]{BChGT}):
\begin{equation}
1-\left|\widehat{\mu}\left(\xi\right)\right|^{2}=\left(\nabla\widehat{\mu}(0)\cdot\xi\right)^{2}-H_{\widehat{\mu}}(0)\xi\cdot\xi+\mathcal{O}\left(|\xi|^{3}\right).\label{eq:Taylor}
\end{equation}
Since $\int_{\mathbb{R}^{d}}xd\mu\left(x\right)=x_{0}=0$, then 
\begin{align*}
\nabla\widehat{\mu}\left(0\right)\cdot\xi= & \left.\nabla\left(\int_{\mathbb{R}^{d}}e^{-2\pi ix\cdot\eta}d\mu\left(x\right)\right)\right|_{\eta=0}\cdot\xi\\
 =&\left(\int_{\mathbb{R}^{d}}\left(-2\pi ix\right)d\mu\left(x\right)\right)\cdot\xi=0.
\end{align*}
For the term with the Hessian matrix we have
\[
-H_{\widehat{\mu}}(0)\xi\cdot\xi=4\pi^{2}\int_{\mathbb{R}^{d}}\left(x\cdot\xi\right)^{2}d\mu\left(x\right).
\]
Assume that $H_{\widehat{\mu}}(0)\xi\cdot\xi=0$ for every $\xi$. 
Since $\mu$ is a probability measure, then, for every $\xi$,
\[
\mu\left(\left\{ x:x\cdot\xi\neq0\right\} \right)=0
\]
and therefore 
\[
\mu\left(\left\{ x\neq0\right\} \right)=0 \ . 
\]
Hence $\mu=\delta_{0}$, against our assumptions. It follows
that there exists $\xi_{0}$, $\left|\xi_{0}\right|=1$, such that
\[
\int_{\mathbb{R}^{d}}\left(x\cdot\xi_{0}\right)^{2}d\mu\left(x\right)>0.
\]
This yields the existence of a cone $\Gamma$ around the
direction $\xi_{0}$ such that for every $\xi\in\Gamma$ we have
\[
\int_{\mathbb{R}^{d}}\left(x\cdot\xi\right)^{2}d\mu\left(x\right)\geq c\left|\xi\right|^{2}.
\]
 By \eqref{eq:Taylor} we have
\begin{equation}
1-\left|\widehat{\mu}\left(\xi\right)\right|^{2}=-H_{\widehat{\mu}}(0)\xi\cdot\xi+\mathcal{O}\left(|\xi|^{3}\right)\geq c_{1}\left|\xi\right|^{2}+\mathcal{O}\left(\left|\xi\right|^{3}\right)\geq\frac{1}{2}c_{1}\left|\xi\right|^{2}\label{eq:Taylor cono}
\end{equation}
 for $\xi\in\Gamma$, $\left|\xi\right|\leq c_{2}$, where $c_{2}$
is a sufficiently small constant. Since
\begin{align*}
D^{2}_{\mu^{\left(\varepsilon\right)}}\left(N\right)= & \ N\sum_{m\neq0}\left(1-\left|\widehat{\mu^{\left(\varepsilon\right)}}\left(m\right)\right|^{2}\right)\left|\widehat{\chi_{rB}}\left(m\right)\right|^{2} \\
&+N^{2}\sum_{h\neq0}\left|\widehat{\chi_{rB}}\left(Mh\right)\right|^{2}\left|\widehat{\mu^{\left(\varepsilon\right)}}\left(Mh\right)\right|^{2}\\
\geq & \ N\sum_{m\neq0}\left(1-\left|\widehat{\mu}\left(\varepsilon m\right)\right|^{2}\right)\left|\widehat{\chi_{rB}}\left(m\right)\right|^{2} \ ,
\end{align*}
then \eqref{eq:Taylor cono} and \eqref{Bes} yield
\begin{align*}
D^{2}_{\mu^{\left(\varepsilon\right)}}\left(N\right)\geq & \ c_{3}N\varepsilon^{2}\sum_{\substack{m\in\Gamma\\
0<\left|\varepsilon m\right|\leq c_{2}
}
}\left[\left|m\right|^{1-d}\cos^{2}\left(2\pi r\left|m\right|-\frac{d+1}{4}\pi\right)\right. \\
&\left.
+\mathcal{O}\left(\left|m\right|^{-d}\right)\right]\\
\geq & \ c_{3}N\varepsilon^{2}\sum_{\substack{m\in\Gamma\\
0<\left|\varepsilon m\right|\leq c_{2}
}
}\left|m\right|^{1-d}\cos^{2}\left(2\pi r\left|m\right|-\frac{d+1}{4}\pi\right)\\
&+\mathcal{O}\left(N\varepsilon^{2}\sum_{\substack{m\in\Gamma\\
0<\left|\varepsilon m\right|\leq c_{2}
}
}\left|m\right|^{-d}\right).
\end{align*}
For the first term, in order to stay away from the zeros of
\[
\cos\left(  2\pi r\left\vert m\right\vert -\frac{d+1}{4}\pi\right),
\]
we introduce, for every fixed $0<w\leq1$, and $K$ large positive integer, the
$d$-dimensional annulus
\[
A_{K,w}=\left\{  t\in\mathbb{R}^{d}:\frac{K+\left(  d+1\right)  /4-w}%
{2r}<\left\vert t\right\vert \leq\frac{K+\left(  d+1\right)  /4+w}%
{2r}\right\}  \ ,
\]
so that \ $\cos^{2}\left(  2\pi r\left\vert m\right\vert -(d+1)\pi/{4}\right)  \geq1/2$ \ whenever $m\in A_{K,1/4}$. Since the width $w/r$ of the
annulus is fixed, then the number of integer points in the annulus is known to be 
asymptotic to its volume. This comes from looking
at the annulus as the difference of two balls, and then observing that for
every $d\geq2$ we have (see e.g. \cite{Kratzel})
\[
K^{d}\left\vert B\right\vert -\operatorname*{card}\left(  KB\cap\mathbb{Z}%
^{d}\right)  =\mathit{o}\left(  K^{d-1}\right)
\]
(for recent results on lattice points in annuli see e.g. \cite{PaSi} or
\cite{CGG}). Therefore, as $K\rightarrow+\infty$, we have
\[
\operatorname*{card}\left(  A_{K,w}\cap\mathbb{Z}^{d}\right)  \sim\left\vert
A_{K,w}\right\vert \ .
\]
Then
\[
\operatorname*{card}\left(  A_{K,1/4}\cap\mathbb{Z}^{d}\right)  \approx
\operatorname*{card}\left(  A_{K,1}\cap\mathbb{Z}^{d}\right)  \ .
\]
Therefore we obtain
\begin{align*}
&\sum_{\substack{m\in\Gamma\\
0<\left|\varepsilon m\right|\leq c_{2}
}
}\left|m\right|^{1-d}\cos^{2}\left(2\pi r\left|m\right|-\frac{d+1}{4}\pi\right)\\
&\geq c_4\sum_{
0<\left|\varepsilon m\right|\leq c_{2}
}\left|m\right|^{1-d}\cos^{2}\left(2\pi r\left|m\right|-\frac{d+1}{4}\pi\right)\\
&\geq
c_5\int_1^{c_2/\varepsilon}dx=\frac{c_6}{\varepsilon}. 
\end{align*}
Hence
\[
D^{2}_{\mu^{\left(\varepsilon\right)}}\left(N\right)\geq c_3c_{6}N\varepsilon+\mathcal{O}\left(N\varepsilon^{2}\log\varepsilon\right)\geq c_{7}N\varepsilon.
\]
Therefore, choosing  $\varepsilon_M=\Omega\left(M^{-1}\right)$ and $\mu_M=\mu^{\left(\varepsilon_M\right)}$  we have
\[
D^{2}_{\mu_M}\left(N\right)=\Omega\left(N^{1-1/d}\right).
\]
\end{proof}
\begin{proof}[Proof of Proposition \ref{prop:no_rescaling}]
Let $0<\alpha<1$, let $t_{0}\in\mathbb{T}^{d}$, with $\varepsilon=\left|t_{0}\right|$,
and let
\[
\mu=\left(1-\alpha\right)\delta_{0}+\alpha\delta_{t_{0}}.
\]
 Then
\begin{align*}
1-\left|\widehat{\mu}\left(m\right)\right|^{2}= & 1-\left|1-\alpha+\alpha e^{2\pi it_{0}\cdot m}\right|^{2}\\
= & 1-\left[1-\alpha+\alpha\cos\left(2\pi t_{0}\cdot m\right)\right]^{2}-\left[\alpha\sin\left(2\pi t_{0}\cdot m\right)\right]^{2}\\
= & 1-\left(1-\alpha\right)^{2}-2\left(1-\alpha\right)\alpha\cos\left(2\pi t_{0}\cdot m\right)-\alpha^{2}\\
= & 4\alpha\left(1-\alpha\right)\sin^{2}\left(\pi t_{0}\cdot m\right).
\end{align*}
By \eqref{discr} and \eqref{fatto} we have
\begin{align*}
 D^2_{\mu}\left(N\right) = & N\sum_{m\neq0}\left(1-\left|\widehat{\mu}\left(m\right)\right|^{2}\right)\left|\widehat{\chi}_{rB}\left(m\right)\right|^{2}\\
&+N^2\sum_{h\neq0}\left|\widehat{\chi}_{rB}\left(Mh\right)\right|^{2}\left|\widehat{\mu}\left(Mh\right)\right|^{2}\\
\leq  & c_1N\alpha\sum_{m\neq0}\left(\sin^{2}\left(\pi t_{0}\cdot m\right)\right)\left|m\right|^{-d-1}+c_2 N^{1-1/d}\\
\leq  &c_1N\alpha\sum_{0<\left|m\right|\leq\varepsilon^{-1}}\varepsilon^{2}\left|m\right|^{-d+1}+c_1N\alpha\sum_{\left|m\right|>\varepsilon^{-1}}\left|m\right|^{-d-1}+c_2 N^{1-1/d}\\
\leq  & c_1 N\alpha\varepsilon+c_2 N^{1-1/d}, 
\end{align*}
with $c_1,c_2$ independent of $\mu,\alpha,N$ and $\varepsilon$. 

Now choose $M^{-1}<\varepsilon_M<1$. Let $\alpha_{M}=\left(M\varepsilon_{M}\right)^{-1}$
and let 
\[\mu_{M}=\left(1-\alpha_{M}\right)\delta_{0}+\alpha_{M}\delta_{t_{M}},\]
where $t_{M}$ is any point in $\mathbb{T}^{d}$ with $\left|t_{M}\right|=\varepsilon_{M}$.
Then
\[
D^{2}_{\mu_{M}}\left(N \right) \leq  c_1 N\left(M\varepsilon_{M}\right)^{-1}\varepsilon_{M}+  c_2 N^{1-1/d}\leq  c_3 N^{1-1/d} \ .
\]
\end{proof}

\end{document}